\documentclass[10pt]{article}
\usepackage{graphics}
\usepackage{graphicx}

\usepackage[a4paper, left=35mm,right=35mm,top=34mm,bottom=34mm]{geometry}
\usepackage[utf8]{inputenc}
\usepackage[T1]{fontenc}
\usepackage[english]{babel}

\usepackage{enumerate}
\usepackage{graphicx}
\usepackage{hyperref}
\hypersetup{
    colorlinks=true,
    linkcolor=blue,
    filecolor=magenta,      
    urlcolor=cyan,
}
\usepackage{lipsum}
\usepackage{amsfonts}
\usepackage{graphicx}
\usepackage{epstopdf}
\usepackage{algorithmic}
\usepackage{lipsum}
\usepackage{amsfonts}
\usepackage{graphicx}
\usepackage{epstopdf}

\usepackage{hyperref}
\ifpdf
  \DeclareGraphicsExtensions{.eps,.pdf,.png,.jpg}
\else
  \DeclareGraphicsExtensions{.eps}
\fi

\usepackage{xcolor}

\usepackage{tikz}

\usepackage{marginnote,cite,amsmath,bbm}
\usepackage{dsfont}
\usepackage{xcolor}
\usepackage{amsmath}
\usepackage{amssymb}
\usepackage{bbm}
\usepackage{psfrag}

\newtheorem{proposition}{\textbf{Proposition}}

\newtheorem{corollary}{\textbf{Corollary}}
\newtheorem{remark}{\textbf{Remark}}
\newtheorem{theorem}{\textbf{Theorem}}
\newtheorem{lemma}{\textbf{Lemma}}

 \def\dx{{\rm d}x}

\def\del  {\partial}
\def\div{\mathrm{div}}
\def\eps{\varepsilon}

\def\u{\mathbf{u}}
\def\R{\mathbb{R}}

\DeclareMathOperator{\ext}{Ext}
\DeclareMathOperator{\supp}{supp}

\DeclareMathOperator{\lo}{lower}
\DeclareMathOperator{\up}{upper}
\DeclareMathOperator{\opt}{opt}
\DeclareMathOperator{\ad}{ad}

\newenvironment{proof}{%
{\noindent \bf Proof : }%
}{%
\hfill$\Box$\\%
}
\def\R{\mathbb{R}}

\usepackage{caption} 
\usepackage{fancyhdr}

\author{
  {\normalsize Michael Hinz}\thanks{Fakult\"{a}t f\"{u}r Mathematik, Universit\"{a}t Bielefeld, Postfach 100131, 33501 Bielefeld,
Germany.}
\and
{\normalsize Fr\'ed\'eric Magoul\`es}\thanks{CentraleSup\'elec, Universit\'e Paris Saclay, France and  University of Pécs, Pécs, Hungary}
		\and
  {\normalsize Anna Rozanova-Pierrat}\thanks{CentraleSup\'elec, Universit\'e Paris-Saclay, France.}
  \and
  {\normalsize Marina Rynkovskaya}\thanks{Peoples’ Friendship University of Russia (RUDN University), Russian Federation}
    \and
	{\normalsize Alexander Teplyaev}\thanks{Department of Mathematics, University of Connecticut, Storrs, CT 06269-3009 USA.}
		}
\title{On the existence of optimal shapes in architecture}
\date{}

\begin{document}
\maketitle
\thispagestyle{fancy}

\begin{abstract}
\noindent We consider shape optimization problems for elasticity systems in architecture. A typical question in this context is to identify a structure of maximal stability close to an initially proposed one. We show the existence of such an optimally shaped structure
within classes of bounded Lipschitz domains and within wider classes of bounded uniform domains with boundaries that may be fractal.
In the first case the optimal shape realizes the infimum of a given energy functional over the class, in the second case it realizes the minimum. As a concrete application we discuss the existence
of maximally stable roof structures under snow loads.
\end{abstract}


\section{Introduction}
The question of how to identify the most stable, the most lightweight or the most suitable (by some parameters) structure is a natural question of shape optimization. The main idea is to vary a shape within a chosen class of domains whose elements all satisfy certain industrial constraints (such as the quantity of material, spatial and dimensional restrictions or some kind of similarity to a fixed form) and to search for an optimal shape in the class that minimizes a given target functional. There is a huge body of literature on different numerical methods~\cite{GERO-1975,MICHALEK-2002,ZANCHETTIN-2011,ZAWIDZKI-2017,ALLAIRE-2007,DAPOGNY-2017} and on implemented packages~\cite{WORTMANN-2017} suitable for simulations.

From the numerical point of view finite dimensional shape optimization problems are always solvable: In a finite dimensional model a discretized domain is determined by a finite mesh, so there are only finitely many possible changes of its boundary which respect all given constraints and one can find at least one shape which is optimal. Non-existence results and other counterexamples usually come from the passage to the limit as the mesh becomes finer and finer. If one abandons the discretized perspective and adopts a more theoretical point of view, where the domains are subsets of $\R^n$ and the target functional is defined on an  infinite dimensional space, the existence of an optimal shape is not trivial at all. There are different classical examples of shape optimization problems for which no optimal shape exists~\cite[Section~4.2]{HENROT-2005}. In particular, the existence of an optimal shape, introduced and formulated for the architecture design framework in~\cite{DAPOGNY-2017}, had so far remained an open problem. Here we solve this problem, and we provide an application to shape optimization for roofs. 

We study a linear elasticity system (formula \eqref{Eq} below) whose solution models the displacement of an architectural structure such as a roof or a bridge. The domain is subject to two types of boundary conditions: a Neumann (or 'stress') boundary condition which models a traction applied to a part of the boundary of the structure, and a homogeneous Dirichlet boundary condition on another part of the boundary which is clamped and therefore cannot be displaced. We refer to~\cite{CIARLET-1988} for a survey on related well-posedness results for bounded Lipschitz domains. Here we update these classical results in the framework of bounded uniform domains whose boundaries support suitable measures, a similar setup is considered in~\cite{HINZ-2020} for the Helmholtz equation. These boundaries may have 'fractal' parts of different Hausdorff dimensions. Fairly general boundary trace and extension results from \cite{AH96, ArendtWarma2003, BIEGERT-2009, JONSSON-2009, WALLIN-1991} apply, and together with Korn's inequality, which remains valid for uniform domains,~\cite{DURAN-2004}, they allow to establish a well-posedness result, Theorem \ref{ThWelPos}. 

The existence of optimal shapes for elliptic problems with homogeneous Dirichlet conditions and homogeneous Neumann conditions is mainly due to Chenais~\cite{CHENAIS-1975}, see~\cite{HENROT-2005} for more details and further references. In this well-known setup the classes of domains for which the existence of optimal shapes is verified are classes of Lipschitz domains contained in a single larger domain (confinement) and having the $\eps$-cone property with the same $\varepsilon>0$. The notions of convergence for domains discussed in this situation are the Hausdorff convergence, the convergence in the sense of characteristic functions and the convergence in the sense of compacts,~\cite{HENROT-2005}. We invoke these classical results to verify the existence of an optimal shape of a roof structure. The assumption of homogeneous Neumann conditions, at least on the part of the boundary that can actually vary in the optimization procedure, is essential for these classical methods to work. From a modeling perspective homogeneous Neumann conditions on the varying parts of the boundary are adequate if there is no significant external force acting on the corresponding parts of the surface of the roof. 

In practice parts of the surface may be subject to significant forces that should be taken into account when searching for an optimal shape. In the simple case of a roof, for instance, the upper surface could be subject to the weight of a heavy load of snow. Such external forces correspond to inhomogeneous Neumann conditions on the varying parts of the boundary. If inhomogeneous Neumann boundary conditions are imposed, care is needed, because boundary integrals appear in the variational formulation of the problem. When varying the shape, the natural notion of convergence for these measures is weak convergence, and this is delicate in the sense that the weak limit of a sequence of codimension one Hausdorff measures (which are the 'natural' surface measures on the boundary of a Lipschitz domain) is not necessarily a Hausdorff measure itself. Using the method in~\cite{MAGOULES-2020} one can show the existence of an optimal shape which realizes the infimum of the energy over a class of bounded Lipschitz domains, Theorem \ref{ThPrincipale}. Employing results from~\cite{HINZ-2020} one can verify the existence of an optimal shape in a larger class of bounded uniform domains 
which then realizes the minimum of the energy over this class, Theorem \ref{ThOptimalEpsDelta}. If the architectural shapes are  recruited from such classes of domains, energy minimizing shapes can be seen to exist, Corollary \ref{C:roof}.

In Section~\ref{SecNotations} we collect some notation. In Section~\ref{SecFuncAn} we introduce the elasticity system, 
discuss trace and extension methods, Theorem \ref{ThTracecheap}, and verify the validity of a norm equivalence with uniform constants, Lemma \ref{L:equivnorms}. We then define weak solutions \eqref{EqVF} and obtain the well-posedness result, Theorem \ref{ThWelPos}.
Shape optimization problems for the simple case of homogeneous Neumann boundary conditions are discussed in Section~\ref{SecShapeOp} using the well-known results from~\cite{CHENAIS-1975} and~\cite{HENROT-2005}. As a practical example we address the existence question for optimal shapes of a roof, Theorem \ref{Throof}, and we comment on a related problem formulated in~\cite{DAPOGNY-2017}. In Section \ref{S:Lip} we follow the method of~\cite{MAGOULES-2020} to study a less simple generalization with possibly inhomogeneous Neumann boundary conditions. Section \ref{S:uni} addresses the uniform domain setup and states the existence of energy minimizing optimal shapes. In Appendix~\ref{AppDef} we provide generalized Green's formulas that justify our definition of weak solution, in Appendix~\ref{S:technical} we briefly comment on a technical detail in the proof of Lemma \ref{L:equivnorms}. Appendix~\ref{SecMosco} contains an auxiliary result, namely the the Mosco convergence of energy functionals for Robin type problems along a sequence of suitably converging uniform domains.

\section*{Acknowledgements}
M.H. gratefully acknowledges financial support by the DFG IRTG 2235 and the DFG CRC 1283. 
A.T. was  supported in part by NSF grant  DMS-1613025. A. R.-P. thanks her student Antoine Verdon for his premilinar work in the subject. M.R. thanks for the financial support  CentraleSup\'elec. 

\section{Some notation}\label{SecNotations}

For the Euclidean scalar product of two vectors $\xi,\eta\in \mathbb{R}^N$ we write $\xi\cdot \eta$. By $\mathcal{M}_N$ we denote the vector space of $N\times N$-matrices with real entries, and by $A:B$ we denote the full contraction
\begin{equation}\label{EqNot:}
	A:B=\mathrm{tr}\: A^tB=\sum_{i,j=1}^N a_{ij} b_{ij}
\end{equation}
of two matrices $A=(a_{ij})_{1\le i,j\le N}$ and $B=(b_{ij})_{1\le i,j\le N}$ from $\mathcal{M}_N$. Note that the full contraction \eqref{EqNot:} provides a natural inner product on $\mathcal{M}_N$. Let $\mathcal{M}_N^s$ be the subspace of $\mathcal{M}_N$ consisting of symmetric matrices and given $\alpha>0$ and $\beta>0$ let $\mathcal{M}_N^s(\alpha,\beta)$ denote the subspace of all invertible $M\in \mathcal{M}_N^s$ such that $M\xi\cdot \xi \ge \alpha |\xi|^2$ and $M^{-1} \xi\cdot\xi \ge \beta|\xi|^2$ for all $\xi \in \R^N$.

Given a domain $\Omega\subset \mathbb{R}^N$ and a vector field 
$\mathbf{v} \in  W^{1,2}(\Omega)^N$ we denote the symmetric part of its gradient by
\[e(\mathbf{v})=\frac{1}{2}\left(\nabla \mathbf{v}+(\nabla \mathbf{v})^t\right). \]
The assumption on $\mathbf{v}$ implies that $e(\mathbf{v})\in L^2(\Omega,\mathcal{M}^s_N)$. Here $W^{1,2}(\Omega)$ is the classical Sobolev space, and product norms on $W^{1,2}(\Omega)^N$ are defined in the natural way. For an element $T=(T_{ij})_{1\le i,j\le N}$ of $W^{1,2}(\Omega)^{N\times N}$ we define $\mathrm{div}\: T\in L^2(\Omega)^N$ as the vector field $(\sum_{j=1}^N \partial_{x_j}T_{1j}, \ldots , \sum_{j=
1}^N \partial_{x_j}T_{Nj})$ .

By $B(x,r)$ we denote the Euclidean open ball centered in $x$ of the radius $r>0$. We use the symbol $\lambda^N$ for the $N$-dimensional Lebesgue measure and the symbol $\mathcal{H}^{N}$ for the $N$-dimensional Hausdorff measure.

\section{Variational formulation and well-posedness}\label{SecFuncAn}

We assume $N\geq 2$. Let $\Omega\subset \mathbb{R}^N$ be a bounded domain and let $\Gamma_{\mathrm{Dir}}$ and $\Gamma_{\mathrm{Neu}}$ be subsets of its boundary $\partial\Omega$. We assume that they are 'almost disjoint' (in the sense that their overlap has zero measure), this will be made precise below. Let $A\in L^\infty(\Omega,\mathcal{M}_N^s(\alpha,\beta))$ and write $\sigma(\mathbf{v})=Ae(\mathbf{v})$, $\mathbf{v} \in  W^{1,2}(\Omega)^N$. We are interested in solutions $\mathbf{u}\in W^{1,2}(\Omega)^N$ to problems of type
\begin{equation}\label{Eq}
\begin{cases}
	-\mathrm{div} \:\sigma(\mathbf{u})&=\mathbf{f} \quad \text{in $\Omega$},\\  
	\hspace{34pt}\mathbf{u}&=0 \quad \text{on $\Gamma_{\mathrm{Dir}}$},\\
	\hspace{8pt}\sigma(\mathbf{u})\cdot n&=\mathbf{g} \quad \text{on $\Gamma_{\mathrm{Neu}}$}. 
	\end{cases}
\end{equation}

Here $n$ denotes the outward unit normal. The domain $\Omega$ is 'clamped' at $\Gamma_{\mathrm{Dir}}$. The part of $\Gamma_{\mathrm{Neu}}$ on which $\mathbf{g}$ is not zero is subjected to a traction represented by the force field $\mathbf{g}$. The vector field $\mathbf{f}$ represents a force field experienced inside $\Omega$. The prospective solution $\mathbf{u}$ is the (unknown) displacement vector field, $e(\mathbf{u})$ is the strain tensor and $\sigma(\mathbf{u})$ is the stress tensor, determined by the given coefficient $A$, which is minus the elasticity tensor in Hooke's law. 


\newpage
\begin{remark}\mbox{}
\begin{enumerate}
\item[(i)]
A more common choice for $\sigma$ within the linear elasticity theory for isotropic materials is $\sigma(\mathbf{v})=\lambda (\mathrm{tr}\: e(\mathbf{v})) I +2 \mu e(\mathbf{v})$, where $\mu>0$ and $\lambda> - 2\mu/N$ are the Lam\'e coefficients of the 
material. See for instance \cite[Section 6.2]{CIARLET-1988}. With insubstantial modifications our results also apply to this case.
\item[(ii)] Note that from ~\eqref{Eq} one can recover the system considered in \cite[Subsection 3.1, formula (1)]{DAPOGNY-2017} if $\Gamma_{\mathrm{Neu}}$ is split into a part where $\mathbf{g}$ is nonzero and one where it is zero.
\end{enumerate}
\end{remark}

We will give a rigorous meaning to ~\eqref{Eq} in terms of a variational formulation. This works well if the boundary $\partial\Omega$ is the support of a suitable measure so that Sobolev functions, fields or tensors on $\Omega$ have well-defined traces on $\partial\Omega$. 

Recall that given $d>0$, a Borel measure  $\mu_F$ on $\mathbb{R}^N$ with $F=\supp \mu_F$ is said to be \emph{upper $d$-regular} if there is a constant $c_d>0$ such that 
\begin{equation}\label{E:upperreg}
\mu_F(B(x,r))\leq c_d r^d,\quad x\in F,\quad 0<r\leq 1.
\end{equation}
This property is well-known and widely used, we refer to \cite{AH96} and \cite{FALCONER-1985}. Recall also that a domain $\Omega\subset \mathbb{R}^N$ is said to be an \emph{$W^{1,2}$-extension domain} if there is a bounded linear extension operator $E:W^{1,2}(\Omega)\to W^{1,2}(\mathbb{R}^N)$, \cite{JONES-1981, ROGERS-2006}. We state a trace result that follows from a special case of \cite[Theorem 5.1]{HINZ-2020} and the finiteness of the measure on $\partial\Omega$. The result is based on \cite[Corollaries 7.3 and 7.4]{BIEGERT-2009}, which in turn use \cite[Theorems 7.2.2 and 7.3.2]{AH96}.

\begin{theorem}\label{ThTracecheap}
Let $\Omega\subset \mathbb{R}^N$ be a bounded $W^{1,2}$-extension domain. Suppose that $\mu_{\partial\Omega}$ is a Borel measure with $\supp \mu_{\partial\Omega}=\partial\Omega$ and such that \eqref{E:upperreg} holds with some $N-2<d\leq N$. 
\begin{enumerate}
\item[(i)]
 There are a compact linear operator $\operatorname{Tr}:W^{1,2}(\Omega) \to L^2(\partial\Omega,\mu_{\partial\Omega})$ and a constant $c_{\mathrm{Tr}}>0$, depending only on $N$, $\varepsilon$, $d$ and $c_d$, such that 
	\[\left\|\operatorname{Tr} f\right\|_{L^2(\partial\Omega,\mu_{\partial\Omega})}\leq c_{\operatorname{Tr}}\left\|f\right\|_{W^{1,2}(\Omega)},\quad  f\in W^{1,2}(\Omega).\]
	Endowed with the norm 
	\[\left\|\varphi\right\|_{\operatorname{Tr}(W^{1,2}(\Omega))}:=\inf\{ \left\|g\right\|_{W^{1,2}(\Omega)}|\ \varphi=\mathrm{Tr}\:g\}\]
	the image $\operatorname{Tr}(W^{1,2}(\Omega))$ becomes a Hilbert space. The embedding $$\operatorname{Tr}(W^{1,2}(\Omega))\subset L^2(\partial\Omega,\mu_{\partial\Omega})$$ is compact.
	\item[(ii)]  There is a linear operator $H_{\partial\Omega}:\operatorname{Tr}(W^{1,2}(\Omega)) \to W^{1,2}(\Omega)$ of norm one such that $\operatorname{Tr}(H_{\partial\Omega}\varphi)=\varphi$ for all $\varphi\in \operatorname{Tr}(W^{1,2}(\Omega))$.
	\end{enumerate}
\end{theorem}

The trace of vector fields or tensors we understand in the component-wise sense.

\begin{remark}\mbox{}
\begin{enumerate}
\item[(i)] The trace operator is defined by $\operatorname{Tr} f:=\widetilde{g}$, where 
\begin{equation}\label{E:pointwiseredef}
\widetilde{g}(x)=\lim_{r\to 0}\frac{1}{\lambda^N(B(x,r))}\int_{B(x,r)}g(y)dy
\end{equation}
is the pointwise redefinition of an extension $g\in W^{1,2}(\mathbb{R}^N)$ of $f$. By the Lebesgue differentiation theorem this limit 
exists outside a $\lambda^N$-null set, but due to the Sobolev regularity of $g$ it exists outside a much smaller set, as can be  made precise using capacities, \cite{AH96, MAZ'JA-1985}. Since $\mu_{\partial\Omega}$ satisfies \eqref{E:upperreg} with $d$ as stated, the set of points of $\partial\Omega$ where this limit exists is of full $\mu_{\partial\Omega}$-measure, \cite[Section 7]{AH96}. The independence of the chosen extension is proved in \cite[Theorem 6.1]{BIEGERT-2009}, another proof is given in \cite[Theorem~1]{WALLIN-1991}. 
\item[(ii)] The extension operator $H_{\partial\Omega}$ is defined as the $1$-harmonic extension: Given $\varphi\in \operatorname{Tr}(W^{1,2}(\Omega))$, $H_{\partial\Omega}\varphi$ is the (unique) element $w$ of $W^{1,2}(\Omega)$ such that $\operatorname{Tr} w=\varphi$ and $\Delta w=w$ in $\Omega$ in the weak sense. Details can be found in \cite[Section~5]{HINZ-2020}. 
\item[(iii)] Clearly the Hausdorff dimension of $\partial\Omega$ is at least $N-1$. If the maximal possible exponent $d$ in \eqref{E:upperreg} is less than $N-1$ then $\mu_{\partial\Omega}$ is singular with respect to the $(N-1)$-dimensional Hausdorff measure $\mathcal{H}^{N-1}$ and assigns positive mass also to some parts of $\partial\Omega$ that have Hausdorff dimension less than $N-1$.
\end{enumerate}
\end{remark}

Let $\Omega$ and $\mu_{\partial\Omega}$ be as in Theorem \ref{ThTracecheap}. For a fixed Borel subset $\Gamma_{\mathrm{Dir}}$ of $\partial\Omega$ we also consider the closed subspace
\[V(\Omega,\Gamma_{\mathrm{Dir}})=\{v\in W^{1,2}(\Omega)|\; \operatorname{Tr} v =0\ \text{$\mu_{\partial\Omega}$-a.e. on $\Gamma_{\mathrm{Dir}}$}\}\]
of $W^{1,2}(\Omega)$.

Now suppose that $\Gamma_{\mathrm{Dir}}$ and $\Gamma_{\mathrm{Neu}}$ are two Borel subsets of $\partial\Omega$ such that $\partial\Omega=\Gamma_{\mathrm{Dir}}\cup \Gamma_{\mathrm{Neu}}$ and $\mu_{\partial\Omega}(\Gamma_{\mathrm{Dir}}\cap \Gamma_{\mathrm{Neu}})=0$. Then a rigorous meaning can be given to the system~\eqref{Eq} in terms of the following variational formulation. We say that a vector field $\u\in V(\Omega,\Gamma_{\mathrm{Dir}})^N$ is a \emph{weak solution} to~\eqref{Eq} with data $\mathbf{f}\in L^2(\Omega)^N$ and $\mathbf{g}\in L^2(\Gamma_{\mathrm{Neu}})^N$ if 
\begin{multline}\label{EqVF}
\int_\Omega \sigma(\u):e(\mathbf{\theta})\dx=\int_{\Omega} \mathbf{f}\cdot \mathbf{\theta} \dx +\int_{\Gamma_{\mathrm{Neu}}} \mathbf{g}\cdot \operatorname{Tr} \mathbf{\theta}\:d \mu_{\partial\Omega}, \quad \mathbf{\theta}\in V(\Omega,\Gamma_{\mathrm{Dir}})^N.
\end{multline}
A generalized Green formula guarantees that in cases when $\Omega$, $A$, $\mathbf{f}$ and $\mathbf{g}$ are smooth and  classical solutions exist, these are also weak solutions, see formula \eqref{EqGreen} of Proposition \ref{P:Greenformula} in Appendix~\ref{AppDef}.

We establish the existence of weak solutions for a specific class of $W^{1,2}$-extension domains. Recall that a domain $\Omega$ of $\R^N$ is an $(\eps,\infty)$-domain, $\eps > 0$, if for any $x,y\in \Omega$ there is a rectifiable arc $\gamma\subset \Omega$ with length $\ell(\gamma)$ joining $x$ to $y$ and satisfying
\begin{enumerate}
 \item $\ell(\gamma)\le \frac{|x-y|}{\eps}$ and
 \item $d(z,\del \Omega)\ge \eps |x-z|\frac{|y-z|}{|x-y|}$ for $z\in \gamma$. 
\end{enumerate}
By ~\cite[Theorem 1]{JONES-1981} any $(\eps,\infty)$-domain is a $W^{1,2}$-extension domain.

%
%

We quote the following special case of the Korn inequality proved in~\cite[Theorem 2.1]{DURAN-2004}. A more standard version for bounded Lipschitz domains can be found in \cite[Theorem 6.3-3]{CIARLET-1988}.
\begin{proposition}
	Let $\Omega\subset \mathbb{R}^n$ be a bounded $(\eps,\infty)$-domain. There is a constant $$C_K(\varepsilon,N,\mathrm{diam}(\Omega))>0$$ depending only on $\eps$, $N$ and $\mathrm{diam} (\Omega)$ such that
	\begin{equation}\label{EqKornG}
	\|\u \|_{W^{1,2}(\Omega)^N}\le C_K(\varepsilon,N,\mathrm{diam}(\Omega))\left(\|\u\|_{L^2(\Omega)^N}+\|e(\u)\|_{L^2(\Omega)^{N\times N}} \right)
\end{equation}
for all $\u \in W^{1,2}(\Omega)^N$. The norm $\u \to \left(\|\u\|_{L^2(\Omega)^N}^2+\|e(\u)\|_{L^2(\Omega)^{N\times N}}^2 \right)^\frac{1}{2}$ is equivalent to $\|\cdot\|_{W^{1,2}(\Omega)^N}$ on $W^{1,2}(\Omega)^N$.
\end{proposition}

If $\Omega$ is a bounded $(\eps,\infty)$-domain then the embedding $W^{1,2}(\Omega)\to L^2(\Omega)$ is compact. (This is actually true for all bounded $W^{1,2}$-extension domains~\cite{ARFI-2017,ROZANOVA-PIERRAT-2020}.) Combining the compactness of this embedding and \eqref{EqKornG} one obtains the following equivalence of the norms on $V(\Omega,\Gamma_{\mathrm{Dir}})^N$. 

\newpage
\begin{lemma}\label{L:equivnorms}\mbox{}
\begin{enumerate}
\item[(i)] Let $\Omega\subset \R^N$ be a bounded $(\eps,\infty)$-domain and $\mu_{\partial\Omega}$ a Borel measure with $\supp \mu_{\partial\Omega}=\partial\Omega$ and such that \eqref{E:upperreg} holds with some $N-2<d\leq N$. Suppose that $\Gamma_{\mathrm{Dir}}$ is a Borel subset of $\partial\Omega$ and $\mu_{\partial\Omega}(\Gamma_{\mathrm{Dir}})>0$. Then there is a constant $c_K(\Omega,\mu_{\partial\Omega},\Gamma_{\mathrm{Dir}})>0$ such that 
	\begin{equation}\label{EqeqivNormV}
		c_K(\Omega,\mu_{\partial\Omega},\Gamma_{\mathrm{Dir}})\: \|\u\|_{W^{1,2}(\Omega)^N}\le \|e(\u)\|_{L^2(\Omega)^{N\times N}}\le \|\u\|_{W^{1,2}(\Omega)^N}
	\end{equation}
	for all $\u \in V(\Omega,\Gamma_{\mathrm{Dir}})^N$.
	\item[(ii)] Let $\Omega\subset D\subset \R^N$ be bounded $(\eps,\infty)$-domains and $\mu_{\partial D}$, $\mu_{\partial\Omega}$ Borel measures with $\supp \mu_{\partial D}=\partial D$ and $\supp \mu_{\partial\Omega}=\partial\Omega$ and such that for both \eqref{E:upperreg} holds with some $N-2<d\leq N$. Suppose that $\Gamma_{\mathrm{Dir}}$ is a Borel subset of $\partial D\cap \partial\Omega$ with $\mu_{\partial D}(\Gamma_{\mathrm{Dir}})>0$ and $\mu_{\partial\Omega}|_{\Gamma_{\mathrm{Dir}}}=\mu_{\partial D}|_{\Gamma_{\mathrm{Dir}}}$. Then the constant $c_K(\Omega,\mu_{\partial\Omega},\Gamma_{\mathrm{Dir}})$ in \eqref{EqeqivNormV} can be replaced by another constant $c_K(\varepsilon, N, D,\mu_{\partial D},\Gamma_{\mathrm{Dir}})>0$ depending only on $\varepsilon$, $N$, $D$, $\mu_{\partial D}$ and $\Gamma_{\mathrm{Dir}}$. Moreover, the Poincar\'e inequality
	\begin{equation}\label{E:Poincare}
	\int_\Omega \u \cdot \u\: \dx\leq C_P(\varepsilon, N, D, \mu_{\partial D},\Gamma_{\mathrm{Dir}})\int_\Omega \nabla \u : \nabla \u\: \dx
	\end{equation}
	holds for all $\u \in V(\Omega,\Gamma_{\mathrm{Dir}})^N$, where $C_P(\varepsilon, N, D, \mu_{\partial D},\Gamma_{\mathrm{Dir}})>0$ is a constant depending only on  $\varepsilon$, $N$, $D$, $\mu_{\partial D}$ and $\Gamma_{\mathrm{Dir}}$.
	\end{enumerate}
\end{lemma}

\begin{proof}
Statement (i) can be proved by the same arguments as used in \cite[Theorem 6.3-4, step (iii) in the proof]{CIARLET-1988}. To see (ii) recall that according to \cite[Theorem 2]{JONES-1981} there exists an extension operator $E_\Omega$ taking $\u$ into a locally integrable function $E_\Omega \u$ on $\mathbb{R}^N$ with gradient in $L^2(\mathbb{R}^N)^{N\times N}$ such that 
\begin{equation}\label{E:extensions}
\left\|\nabla E_\Omega \u\right\|_{L^2(D)^{N\times N}}\leq \left\|\nabla E_\Omega \u\right\|_{L^2(\mathbb{R}^N)^{N\times N}}
\leq C_{\ext}(\varepsilon, N)\left\|\nabla \u\right\|_{L^2(\Omega)^{N\times N}},
\end{equation}
where $C_{\ext}(\varepsilon, N)$ is a constant depending only on $\varepsilon$ and $N$. By adding two inequalities we can see the same holds with $e(\u)$ in place of $\nabla \u$. The pointwise redefinition $\widetilde{E_\Omega \u}$ of $E_\Omega \u$ in the sense of \eqref{E:pointwiseredef} satisfies 
\begin{equation}\label{E:itreallyvanishes}
\widetilde{E_\Omega \u}=0\quad \text{$\mu_{\partial D}$-a.e. on $\Gamma_{\mathrm{Dir}}$,}
\end{equation} 
see Appendix~\ref{S:technical} for details of the brief argument. Using \eqref{E:itreallyvanishes} and the boundedness of $D$ one can then follow a standard pattern, see \cite[Proposition 7.1]{Egert2015} or \cite{EVANS-2010}, to obtain the Poincar\'e inequality
\begin{equation}\label{E:PD}
\int_D \ext_\Omega \u \cdot \ext_\Omega\u\: \dx\leq C_P(\varepsilon, N, D, \mu_{\partial D},\Gamma_{\mathrm{Dir}})\int_D \nabla \ext_\Omega \u : \nabla \ext_\Omega \u\: \dx
\end{equation}
for all $\u \in V(\Omega,\Gamma_{\mathrm{Dir}})^N$. An application of (i) to $D$ yields 
\begin{align}
c_K(D,\mu_{\partial D},\Gamma_{\mathrm{Dir}})\left\|\nabla \ext_\Omega\u\right\|_{L^2(D)^{N\times N}}&\leq \left\|e(\ext_\Omega\u)\right\|_{L^2(D)^{N\times N}}\notag\\
&\leq C_{\ext}(\varepsilon, N)\left\|e(\u)\right\|_{L^2(\Omega)^{N\times N}},\notag
\end{align}
and combining with \eqref{E:PD}, we arrive at
\[\left\|\u \right\|_{L^2(\Omega)^N}\leq \frac{C_P(D, \mu_{\partial D},\Gamma_{\mathrm{Dir}})^{1/2} C_{\ext}(\varepsilon, N)}{c_K(D,\mu_{\partial D},\Gamma_{\mathrm{Dir}})}\left\|e(\u)\right\|_{L^2(\Omega)^{N\times N}}.\]
Plugging this into \eqref{EqKornG}, we obtain the first claim in (ii). Slight modifications of the preceding arguments also give \eqref{E:Poincare}.
\end{proof}



We obtain the following well-posedness result:
\begin{theorem}\label{ThWelPos}

Let $\Omega\subset \mathbb{R}^N$ be a bounded $(\varepsilon,\infty)$-extension domain. Suppose that $\mu_{\partial\Omega}$ is a Borel measure with $\supp \mu_{\partial\Omega}=\partial\Omega$ and such that \eqref{E:upperreg} holds with some $N-2<d\leq N$. Suppose that $\Gamma_{\mathrm{Dir}}$ and $\Gamma_{\mathrm{Neu}}$ are Borel subsets of $\partial\Omega$ such that $\partial\Omega=\Gamma_{\mathrm{Dir}}\cup \Gamma_{\mathrm{Neu}}$, $\mu_{\partial\Omega}(\Gamma_{\mathrm{Dir}}\cap \Gamma_{\mathrm{Neu}})=0$ and $\mu_{\partial\Omega}(\Gamma_{\mathrm{Dir}})>0$.

 Then for all $\mathbf{g}\in [\operatorname{Tr}(W^{1,2}(\Omega))]^N$ and $\mathbf{f}\in L^2(\Omega)^N$ there is a unique weak solution $\u\in V(\Omega,\Gamma_{\mathrm{Dir}})^N$ of~\eqref{EqVF}. Moreover, there is a constant $C>0$ such that
	\begin{equation}
		\label{EqEstimAPriori}
		\|\u\|_{W^{1,2}(\Omega)^N}\le C\left(\|\mathbf{f}\|_{L^2(\Omega)^N}+\|\mathbf{g}\|_{[\operatorname{Tr}(W^{1,2}(\Omega))]^N} \right) 
	\end{equation}

\end{theorem}

\begin{proof} 
	Theorem \ref{ThWelPos} is a direct consequence of the Lax-Milgram Lemma. Indeed, by the definition of $\sigma(\u)$ and \eqref{EqeqivNormV},
	\[\int_\Omega \sigma(\mathbf{v}):e(\mathbf{v})\dx\ge \alpha \int_\Omega e(\mathbf{v}):e(\mathbf{v})\dx\ge \alpha c^{-1} \|\mathbf{v}\|_{W^{1,2}(\Omega)^N}^2,\quad \mathbf{v}\in V(\Omega,\Gamma_{\mathrm{Dir}})^N.\]
	Also the continuity of the bilinear form is seen easily. The continuity of the linear functional $\theta\mapsto \int_{\Omega} \mathbf{f}\cdot \mathbf{\theta} \dx+ \int_{\Gamma_{\mathrm{Neu}}} \mathbf{g}\cdot \operatorname{Tr} \mathbf{\theta}\:d \mu_{\partial\Omega}$ on $V(\Omega,\Gamma_{\mathrm{Dir}})^N$
	follows using Theorem \ref{ThTracecheap}, and by the same theorem we obtain  estimate~\eqref{EqEstimAPriori}.
\end{proof}

\begin{remark}\label{R:constantsindep}
If $\Omega$ in Theorem \ref{ThWelPos} satisfies the hypotheses of Lemma \ref{L:equivnorms} (ii) then \eqref{EqEstimAPriori} holds with a constant $C>0$ depending only on $\varepsilon$, $N$, $D$, $\mu_{\partial D}$ and $\Gamma_{\mathrm{Dir}}$.
\end{remark}

\section{Lipschitz optimal shapes for homogeneous Neumann conditions}\label{SecShapeOp}

We proceed to the existence of optimal shapes for the elasticity system \eqref{Eq}. In this section we prove it for a practical example within a well-known setup involving Lipschitz domains and explain possible generalizations and the difficulties involved. This may be seen as motivation for our new existence results in Sections \ref{S:Lip} and \ref{S:uni}.

For $A$ as in \eqref{Eq}, and $\Omega$, $\mu_{\partial\Omega}$, $\Gamma_{\mathrm{Dir}}$, $\Gamma_{\mathrm{Neu}}$ as in Theorem \ref{ThWelPos} we can  define the functional 
\begin{equation}\label{E:functional}
J(\Omega, \mu_{\partial\Omega}, \mathbf{v}):=c_1\int_\Omega|\mathbf{v}|^2\dx+c_2\int_\Omega Ae(\mathbf{v}):e(\mathbf{v})\dx, \quad \mathbf{v}\in V(\Omega,\Gamma_{\mathrm{Dir}})^N,
\end{equation}
where $c_1$ and $c_2$ are fixed nonnegative constants. Suppose that $\mathbf{f}$ and $\mathbf{g}$ are fixed data, that we can find a class of domains within which $\Omega$ can vary but $\Gamma_{\mathrm{Dir}}$ is kept fixed and that for each $\Omega$ from that class Theorem \ref{ThWelPos} yields a unique 
weak solution $\u(\Omega,\mu_{\partial\Omega})$ to \eqref{Eq} with the same data  $\mathbf{f}$ and $\mathbf{g}$. Then we may interpret
\begin{equation}\label{EqJ}
\Omega\mapsto J(\Omega,\mu_{\partial\Omega},\u(\Omega,\mu_{\partial\Omega}))
\end{equation}
as a functional on this class of domains, and we may attempt to minimize it.

\begin{remark}\mbox{}
\begin{enumerate}
\item[(i)] For $c_1=0$ and $c_2=1$ the value $J(\Omega,\mu_{\partial\Omega},\u(\Omega,\mu_{\partial\Omega}))$ is the elastic energy stored in $\Omega$ (the compliance of $\Omega$), \cite[Subsection 3.1]{DAPOGNY-2017}. For $c_1=1$ and $c_2=0$ the value $J(\Omega,\mu_{\partial\Omega},\u(\Omega, \mu_{\partial\Omega}))$ is the square of the $L^2$-norm of the displacement field, and the minimization of this norm may be viewed as a mathematically tractable substitute for the (intractable) minimization of the $L^\infty$-norm.
\item[(ii)] For a mixed boundary value problem for the Helmholtz equation functionals similar to \eqref{EqJ}, but with an additional boundary term, are studied in \cite[Section 7]{HINZ-2020}.  
\end{enumerate}
\end{remark}

We discuss a practical application. Suppose that we are searching for the most stable (strongest) roof for a building and model it by a domain $\Omega\subset \mathbb{R}^3$. We assume that the moving part $\Gamma_{\mathrm{Neu}}$ of the boundary $\partial\Omega$ is the union of two disjoint parts $\Gamma_{\up}$ and $\Gamma_{\lo}$ which model the upper and the lower side of the roof. The set $\Gamma_{\mathrm{Dir}}$ models vertical boundary parts of the roof at which it is fixed by a structural connection to other parts of the building. See Fig.~\ref{FigRoof}. We impose the condition that the volume $c_v>0$ of the roof itself remains fixed and that $\Gamma_{\up}$ and $\Gamma_{\lo}$ are always parallel with a vertical distance 
\begin{equation}
h_z(\Omega)=\frac{c_v}{\mu(\Gamma_{\lo)}}.
\end{equation}
Since $c_v$ is kept constant, $h_z(\Omega)$ will typically change when $\Gamma_{\up}$ and $\Gamma_{\lo}$ move. We assume that $\mathbf{g}=0$ and set 
\begin{equation}\label{Eqhz}
\mathbf{f}=\rho_0 h_z(\Omega)\mathbf{e}_z, 
\end{equation}
where $\rho_0>0$ is a constant (mass) density and $\mathbf{e}_z=(0,0,1)$. The field $\mathbf{f}(\Omega)$ represents the weight of the roof $\Omega$.

\begin{figure}[!htb] 
\begin{center}
\psfrag{O}{$\Omega$}
\psfrag{G}{$G$}
\psfrag{D}{$D$}
\psfrag{Gu}{$\Gamma_{\up}$}
\psfrag{Gl}{$\Gamma_{\lo}$}
\psfrag{hz}{$h_z$}
\psfrag{Gd}{$\Gamma_{\mathrm{Dir}}$}
\includegraphics[width=8cm]{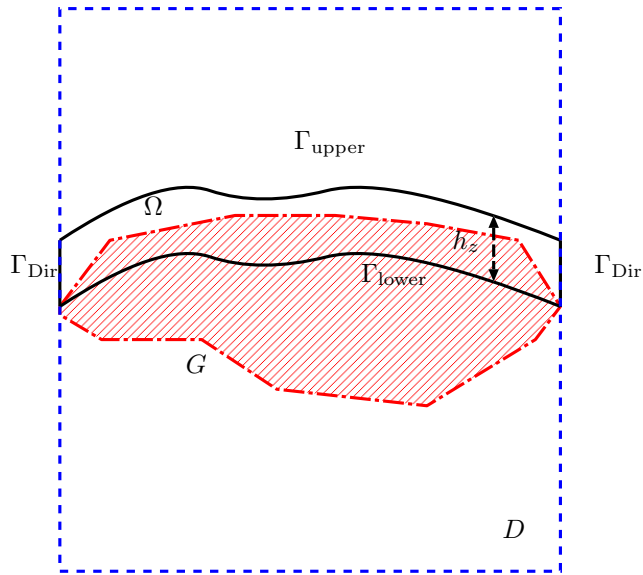}
\end{center}
\caption{\label{FigRoof} The boundary parts $\Gamma_{\up}$ and $\Gamma_{\lo}$ remain parallel to each other with distance $h_z$ as defined in~\eqref{Eqhz}. The lower part $\Gamma_{\lo}$ is supposed to stay within the closure of a fixed open set $G$. We assume that all possible shapes $\Omega$ (and therefore also the set $G$) are subsets of a fixed open set $D$.} 
	\end{figure}

To introduce suitable classes of domains let us recall the following from ~\cite{AGMON-1965,CHENAIS-1975}. Given $\varepsilon>0$, a domain $\Omega\subset\mathbb{R}^N$ is said to have the \emph{$\eps$-cone property} if 
for all $x\in \partial\Omega$ there exists $\xi_x\in \R^{N}$  with  $\|\xi_x\|=1$  such that for all $y \in \overline{\Omega}\cap B(x,\eps)$
$$  C(y,\xi_x,\eps)=\{z\in \R^{N}| (z-y,\xi_x)\ge \cos(\eps)\|z-y\| \hbox{ and } 0<\|z-y\|<\eps\}\subset \Omega.$$
It is well-known that a domain $\Omega$ with bounded boundary $\partial\Omega$ has the $\eps$-cone property for some $\varepsilon>0$ if and only if it is a Lipschitz domain, \cite[Theorem 2.4.7]{HENROT-2005}. 
Now let $D\subset \mathbb{R}^3$ be a bounded Lipschitz domain and $\varepsilon >0$. Somewhat similarly to \cite[Section 2.4]{HENROT-2005} we write 
\begin{equation}\label{E:Oeps}
\mathcal{O}(D,\varepsilon):=\{\Omega \subset D\ |\ \text{$\Omega$ is a domain satisfying the $\varepsilon$-cone property}\}.
\end{equation}
Let $G$ be a nonempty open proper subset of $D$, $\Gamma_{\mathrm{Dir}}$ a compact subset of $\partial D$ with $\mathcal{H}^2(\Gamma_{\mathrm{Dir}})>0$, let $\varepsilon>0$, $c_v>0$ and $0<\ell_0<\ell_1$. We define a class of admissible shapes by   
\begin{multline}\label{EqUadClass}
	U_{\ad}(D,G,\varepsilon,c_v,\ell_0,\ell_1, \Gamma_{\mathrm{Dir}})
	=\{\Omega\in \mathcal{O}(D,\varepsilon)\ |\ \lambda^3(\Omega)=c_\nu, \ \ell_0\le \mathcal{H}^{2}(\partial\Omega)\le \ell_1,\\	
	\partial\Omega=\Gamma_{\mathrm{Dir}}\cup \Gamma_{\mathrm{Neu}},\ \mathcal{H}^2(\Gamma_{\mathrm{Dir}}\cap \Gamma_{\mathrm{Neu}})=0,\ \Gamma_{\mathrm{Dir}}= \partial \Omega\cap \partial D, \\	
	\Gamma_{\mathrm{Neu}}=\Gamma_{\lo}\cup \Gamma_{\up},\ \Gamma_{\lo}\subset \overline{G}, \ \Gamma_{\up}=\Gamma_{\lo}+h_z\mathbf{e}_z\}.
\end{multline}

In this case the existence of an optimal shape realizing the minimum of the functional on $U_{\ad}(D,G,\varepsilon,c_v,\ell_0,\ell_1, \Gamma_{\mathrm{Dir}})$ follows from the results of  Chenais~\cite{CHENAIS-1975}. By  \cite[Theorem 2.4.10]{HENROT-2005} the collection $\mathcal{O}(D,\varepsilon)$ of domains having the $\varepsilon$-cone property and contained in $D$ is compact with respect to the convergence in the sense of characteristic functions, the convergence in the Hausdorff sense and the convergence in the sense of compacts. For any sequence $\Omega_n\to \Omega$ of sets in this class coverging in all three senses, their boundaries $\del \Omega_n \to \del \Omega$ and also their closures $\overline{\Omega}_n\to \overline{\Omega}$ converge in the Hausdorff sense. The set of domains $U_{ad}(D,G,\varepsilon,c_v,\ell_0,\ell_1, \Gamma_{\mathrm{Dir}})$ defined in~\eqref{EqUadClass} is a closed subset of the class $\mathcal{O}(D,\varepsilon)$ and therefore compact. Note also that each $\Omega\in U_{\ad}(D,G,\varepsilon,c_v,\ell_0,\ell_1, \Gamma_{\mathrm{Dir}})$, together with 
$\mu_{\partial\Omega}=\mathcal{H}^{2}|_{\partial\Omega}$, satisfies the hypotheses of Theorem \ref{ThWelPos}. Proceeding as in~\cite[Theorem~4.3.1]{HENROT-2005} we obtain the following result for the case of zero Neumann data.

\begin{theorem}\label{Throof}
Let $U_{ad}(D,G,\varepsilon,c_v,\ell_0,\ell_1, \Gamma_{\mathrm{Dir}})$ be as in \eqref{EqUadClass}, $\rho_0>0$, $c_1\geq 0$ and $c_2\geq 0$. For each $\Omega\in  U_{ad}(D,G,\varepsilon,c_v,\ell_0,\ell_1, \Gamma_{\mathrm{Dir}})$ let $\u(\Omega, \mathcal{H}^2)$ denote the unique weak solution to 
\eqref{Eq} on $\Omega$ with $\mathbf{g}\equiv 0$ and $\mathbf{f}\equiv \mathbf{f}(\Omega)$ defined as in \eqref{Eqhz}. 

Then 
there is an optimal shape $\Omega_{\opt}\in U_{ad}(D,G,\varepsilon,c_v,\ell_0,\ell_1, \Gamma_{\mathrm{Dir}})$ realizing the minimum of~\eqref{EqJ} over $U_{ad}(D,G,\varepsilon,c_v,\ell_0,\ell_1, \Gamma_{\mathrm{Dir}})$.
\end{theorem}

\begin{remark}\mbox{}
\begin{enumerate}
\item[(i)] A similar example of a shape optimization problem is studied in~\cite{DAPOGNY-2017}. There the authors propose to minimize \eqref{EqJ} with $c_1=0$ and $c_2=1$ in \eqref{E:functional} in a setup where $\Gamma_{\mathrm{Dir}}$ and a part $\Gamma_{\mathrm{Neu},\mathrm{fix}}$ of $\Gamma_{\mathrm{Neu}}$ are fixed and only $\Gamma:=\Gamma_{\mathrm{Neu}}\setminus \Gamma_{\mathrm{Neu},\mathrm{fix}}$, on which $\mathbf{g}$ is assumed to be zero, can move. On the fixed part $\Gamma_{\mathrm{Neu},\mathrm{fix}}$ the function $\mathbf{g}$ can be nonzero. They keep the volume of $\Omega$ constant and introduce a penalization term in the functional which restricts the area of the possible changes of $\Gamma$. This term is based on the (signed) distance to an initially given shape $\Gamma_0$. 
\item[(ii)] The idea to restrict the possible shapes to a neigborhood of a given shape is equivalent to fixing a compact set $\overline{G}$ within  which the moving boundary part is required to stay, $\Gamma\subset \overline{G}$. In ~\cite{MAGOULES-2020} the latter condition is shown to be necessary to ensure the existence on an optimal shape. A comparable condition is also used in ~\cite{HINZ-2020}, where the domains $\Omega$ are required to contain another fixed open set $D_0$.
 \end{enumerate}
\end{remark}

The sketched roof optimization problem can easily be generalized in the sense that \eqref{EqJ} could be minimized over
\begin{multline}\label{EqUadClassDap}
	U_{\ad}(D,G,\varepsilon,c_v,\Gamma_{\mathrm{Dir}}, \Gamma_{\mathrm{Neu},\mathrm{fix}})=\{\Omega\in \mathcal{O}(D,\varepsilon)\ | \ \lambda^N(\Omega)=c_v, \\
	\partial\Omega=\Gamma_{\mathrm{Dir}}\cup \Gamma_{\mathrm{Neu}},\ \mathcal{H}^{N-1}(\Gamma_{\mathrm{Dir}}\cap \Gamma_{\mathrm{Neu}})=0,\ \Gamma_{\mathrm{Dir}}= \partial \Omega\cap \partial D, \\
	\Gamma_{\mathrm{Neu},\mathrm{fix}}\subset \Gamma_{\mathrm{Neu}},\ \Gamma_{\mathrm{Neu}}\setminus \Gamma_{\mathrm{Neu},\mathrm{fix}}\subset \overline{G}\},
\end{multline}
where $D$, $G$, $\varepsilon$, $c_v$, $\Gamma_{\mathrm{Dir}}$ are similarly as before. In this setup $\Gamma_{\mathrm{Neu},\mathrm{fix}} \subset \Gamma_{\mathrm{Neu}}$ is kept fixed and only $\Gamma=\Gamma_{\mathrm{Neu}}\setminus \Gamma_{\mathrm{Neu},\mathrm{fix}}$ can move, as in~\cite{DAPOGNY-2017}. If $\mathbf{g}$ can be nonzero only on the fixed part $\Gamma_{\mathrm{Neu},\mathrm{fix}}$ (but not on the moving part $\Gamma$) then one can still prove the existence of an optimal shape by the method of \cite[Theorem 4.3.1]{HENROT-2005}.

As mentioned in the introduction, it may be desirable to allow nonzero Neumann data $\mathbf{g}$ also on the moving part $\Gamma$ of the boundary (for instance, if we wish to include an external force on the roof surface caused by the weight of snow).
In the case of nonzero Neumann data we need to follow our recent studies~\cite{MAGOULES-2020,HINZ-2020} to obtain results on compactness and on the existence of optimal shapes. The main additional difficulty consists in the fact that nonzero Neumann data on the moving part $\Gamma$ requires a correct handling of the 'area measure' $\mu_{\partial\Omega}|_\Gamma$ on this part, because the variational formulation \eqref{EqVF} involves integrals with respect to it. A sequence of domains from \eqref{EqUadClassDap} converges to a domain in that class in the three senses discussed above. However, the natural type of convergence of measures on the boundaries is their weak$^\ast$ convergence, and the weak$^\ast$ limit of Hausdorff measures is not necessarily a Hausdorff measure, ~\cite[Section 3]{MAGOULES-2020}. 

In the next two sections we present two results regarding this issue. The first, Theorem \ref{ThPrincipale}, states that 
for classes of admissible shapes built on $\mathcal{O}(D,\varepsilon)$ we can find optimal shapes that realize the infimum of the functional \eqref{EqJ}. The second, Theorem \ref{ThOptimalEpsDelta}, considers wider parametrized classes of admissible shapes, namely certain uniform domains whose boundaries carry measures satisfying \eqref{E:upperreg}. It guarantees the existence of optimal shapes which then indeed realize the minima of \eqref{EqJ} over such classes. This may be viewed as a relaxation of the Lipschitz boundary case.


\section{Lipschitz optimal shapes realizing the infimum of the energy}\label{S:Lip}

We follow the method in~\cite{MAGOULES-2020}. Given a bounded Lipschitz domain $D\subset \mathbb{R}^N$ and $\varepsilon>0$  we write, similarly as before, $\mathcal{O}(D,\varepsilon)$ for the collection of all domains $\Omega\subset D$ that satisfy the $\varepsilon$-cone property. Now suppose that $G$ is a nonempty open proper subset of $D$, $\Gamma_{\mathrm{Dir}}$ a compact subset of $\partial D$ with $\mathcal{H}^{N-1}(\Gamma_{\mathrm{Dir}})>0$ and $\varepsilon>0$, $c_v>0$, $\hat{c}>0$ are given constants. We define the class 
\begin{multline}\label{E:theclass}
	U_{\ad}(D,G,\varepsilon,c_v, \hat{c}, \Gamma_{\mathrm{Dir}})=\{\Omega\in \mathcal{O}(D,\varepsilon)\ | \ \lambda^N(\Omega)=c_v, \\
	\partial\Omega=\Gamma_{\mathrm{Dir}}\cup \Gamma_{\mathrm{Neu}},\ \mathcal{H}^{N-1}(\Gamma_{\mathrm{Dir}}\cap \Gamma_{\mathrm{Neu}})=0,\ \Gamma_{\mathrm{Dir}}= \partial \Omega\cap \partial D, \Gamma_{\mathrm{Neu}}\subset \overline{G}\\
	\text{ and $\mathcal{H}^{N-1}(\Gamma_{\mathrm{Neu}} \cap B(x,r)) \le \hat{c}r^{N-1}$ for any $x\in \Gamma_{\mathrm{Neu}}$ and $r>0$}\}.
\end{multline}

\begin{remark}
Since $D$ is bounded there is some $\ell_1>0$, depending on $\hat{c}>0$, such that for all $\Omega \in U_{\ad}(D,G,\varepsilon,c_v, \hat{c}, \Gamma_{\mathrm{Dir}})$ we have $\mathcal{H}^{N-1}(\Gamma_{\mathrm{Neu}})\le \ell_1$, and a smaller $\hat{c}$ leads to a smaller bound $\ell_1$. Clearly there is also some uniform lower bound $\ell_0$ for $\mathcal{H}^{N-1}(\Gamma_{\mathrm{Neu}})$, and consequently $\hat{c}$ must be chosen large enough to ensure the class $\Omega \in U_{\ad}(D,G,\varepsilon,c_v, \hat{c}, \Gamma_{\mathrm{Dir}})$ contains more than one domain.
\end{remark}

Now suppose that $\mathbf{f}\in L_2(D)^N$ and $\mathbf{g}\in  [\mathrm{Tr}(W^{1,2}(D))]^N$ are given. For any domain $\Omega$ and Borel measure $\mu_{\partial\Omega}$ on $\partial\Omega$ satisfying the hypotheses of Theorem \ref{ThWelPos}, let $\u(\Omega, \mu_{\partial\Omega})$ denote the unique weak solution for \eqref{Eq} on $\Omega$ with $\mu_{\partial\Omega}$ on $\partial\Omega$. Note that in particular each 
$\Omega \in U_{\ad}(D,G,\varepsilon,c_v, \hat{c}, \Gamma_{\mathrm{Dir}})$ satisfies these hypotheses with $\mu_{\partial\Omega}=\mathcal{H}^{N-1}|_{\partial\Omega}$, and similarly as before we write $\u(\Omega,\mathcal{H}^{N-1})$ for the corresponding weak solution.

We consider the functional~\eqref{EqJ} with $c_1\geq 0$ and $c_2\geq 0$ in \eqref{E:functional}. Ideally we would like to minimize $J(\Omega,\mathcal{H}^{N-1},\u(\Omega,\mathcal{H}^{N-1}))$ on $U_{\ad}(D,G,\varepsilon,c_v, \hat{c}, \Gamma_{\mathrm{Dir}})$. However, under the present hypotheses one can only prove the following.

%
%
 
  \begin{theorem}\label{ThPrincipale}
  Let $U_{\ad}(D,G,\varepsilon,c_v, \hat{c}, \Gamma_{\mathrm{Dir}})$ be as in \eqref{E:theclass}, $c_1\geq 0$, $c_2\geq 0$, $\mathbf{f}\in L_2(D)^N$ and $\mathbf{g}\in  [\mathrm{Tr}(W^{1,2}(D))]^N$. 

Then there are a domain $\Omega_{\opt}\in U_{\ad}(D,G,\varepsilon,c_v, \hat{c}, \Gamma_{\mathrm{Dir}})$ and a finite Borel measure $\mu_{\partial\Omega_{\opt}}$ on $\partial\Omega_{\opt}$, equivalent to $\mathcal{H}^{N-1}|\partial\Omega_{\opt}$, such that $\mathcal{H}^{N-1}(\Gamma_{\mathrm{Neu},\opt})\leq \mu_{\partial\Omega_{\opt}}(\Gamma_{\mathrm{Neu},\opt})$
 and
\begin{multline*}
J(\Omega_{\opt}, \mu_{\partial\Omega_{\opt}}, \u(\Omega_{\opt},\mu_{\partial\Omega_{\opt}}))\\
=\inf_{\Omega\in U_{\ad}(D,G,\varepsilon,c_v, \hat{c}, \Gamma_{\mathrm{Dir}})}J(\Omega, \mathcal{H}^{N-1}, \u(\Omega,\mathcal{H}^{N-1})).
\end{multline*}

If $\mu_{\partial\Omega_{\opt}}=\mathcal{H}^{N-1}|_{\partial\Omega_{\opt}}$, then the infimum is actually a minimum.

\end{theorem}

\begin{proof} Insubstantial modifications of \cite[Lemma~3.1]{MAGOULES-2020} guarantee that for fixed parameters the class $U_{\ad}(D,G,\varepsilon,c_v, \hat{c}, \Gamma_{\mathrm{Dir}})$ is compact with respect to the convergence in the sense of characteristic functions, the Hausdorff sense and the sense of compacts. Moreover, each sequence $(\Omega_n)_n$ in $U_{\ad}(D,G,\varepsilon,c_v, \hat{c}, \Gamma_{\mathrm{Dir}})$ has a subsequence $(\Omega_{n_k})_k$ converging to a limit domain $\Omega_\ast$ in the Hausdorff sense, the sense of characteristic functions and the sense of compacts and for which the Hausdorff measures $\mu_{\partial\Omega_{n_k}}=\mathcal{H}^{N-1}|_{\partial\Omega_{n_k}}$ converge weakly$^*$ to a measure $\mu_{\partial\Omega_\ast}$ which is equivalent to $\mathcal{H}^{N-1}|_{\partial\Omega_\ast}$ and
satisfies 
\begin{equation}\label{E:inequality}
\mathcal{H}^{N-1}(\Gamma_{\mathrm{Neu},\ast}) \le \mu_{\partial\Omega_\ast}(\Gamma_{\mathrm{Neu},\ast}).
\end{equation} 
Here we write $\Gamma_{\mathrm{Neu},\ast}$ for the part of $\partial\Omega_\ast$ that corresponds to $\Gamma_{\mathrm{Neu}}$ in \eqref{E:theclass}, below we will use a similar notation with analogous meaning.

Now suppose that this sequence $(\Omega_n)_{n}\subset \hat{U}_{ad}(D, G, \eps, \hat{c}, c_v,\ell_1)$ is a minimizing sequence of the functional \eqref{EqJ}, this minimizing sequence exists since the functional is nonnegative. We relabel and denote the subsequence $(\Omega_{n_k})_k$ above by $(\Omega_{k})_k$. Let $\u_k:=\u(\Omega_k,\mathcal{H}^{N-1})$ denote the unique weak solution to \eqref{Eq} on $\Omega_k$ and let $\u_\ast:=\u(\Omega_\ast,\mu_{\partial\Omega_\ast})$ denote that on $\Omega_\ast$ with measure $\mu_{\partial\Omega_\ast}$ on the boundary.

Since all domains $\Omega_k$ have Lipschitz boundaries of finite Hausdorff measures, there are extension operators $E: W^{1,2}(\Omega_k)^N\to W^{1,2}(\R^{N})^N$ and a constant $C_E>0$ independent of $k$ such that $\|E \u_k \|_{W^{1,2}(D)^N}\le C_E\|\u_k\|_{W^{1,2}(\Omega_k)^N}$ for all $k$, see~\cite{CHENAIS-1975}. By Lemma \ref{L:equivnorms} (ii) and Theorem \ref{ThWelPos} (together with Theorem \ref{ThTracecheap}) we can find a constant $C>0$, independent of $k$, such that 
\[  \sup_k\|\u_k\|_{W^{1,2}(\Omega_k)^N} \leq C \left(\|\mathbf{f}\|_{L_2(D)^N}+\|\mathbf{g}\|_{\mathrm{Tr}(W^{1,2}(D))]^N}\right).  \]
This means that $(E \u_k |_{D})_{k}$ is bounded in $W^{1,2}(D)^N$, and consequently there exists $\u^*\in W^{1,2}(D)^N$ such that $E \u_k |_{D}\rightharpoonup \u^*$ in $W^{1,2}(D)^N$.

We define linear functionals $F_{k}$ and $F$ on $V(\Omega,\Gamma_{\mathrm{Dir}})^N$ by 
\[F_{k}[\phi]=\int_{\Omega_k} \sigma(\u_k):e(\mathbf{\phi})\dx-\int_{\Omega_k} \mathbf{f}\cdot \mathbf{\phi} -\int_{\Gamma_{\mathrm{Neu},k}} \mathbf{g}\cdot \mathbf{\phi}d \mathcal{H}^{N-1},\]
$\mathbf{\phi}\in V(\Omega,\Gamma_{\mathrm{Dir}})^N$, and 
\[	F[\phi]=\int_{\Omega} \sigma(\u^*):e(\mathbf{\phi})\dx-\int_{\Omega_\ast} \mathbf{f}\cdot \mathbf{\phi} -\int_{\Gamma_{\mathrm{Neu},\ast}} \mathbf{g}\cdot \mathbf{\phi}d \mu^*.\]
Following the proof of~\cite[Theorem~3.2]{MAGOULES-2020} one can see that $F[\phi]=\lim_k F_{k}[\phi]=0$ for all $\phi\in V(D,\Gamma_{\mathrm{Dir}})^N$, which implies that $F[\phi]=0$, $\phi\in V(D,\Gamma_{\mathrm{Dir}})^N$. By the uniqueness established in Theorem \ref{ThWelPos} we must have $\u^*\vert_{\Omega_\ast}=\u_\ast$.

Using \eqref{EqVF} it is not difficult to show that 
\[\lim_k\left\|\mathds{1}_{\Omega_k}\nabla E\u_k\right\|_{L^2(D)^{N\times N}}=\left\|\mathds{1}_{\Omega}\nabla E\u^\ast\right\|_{L^2(D)^{N\times N}}.\] %
Since by the Rellich-Kondrachov theorem also $\lim_k E\u_k=\u^\ast$ in $L^2(D)^N$, it follows that $\lim_{k}J(\Omega_k,\mathcal{H}^{N-1},\u_k)=J(\Omega_\ast,\mu_{\partial\Omega_\ast},\u_\ast)$, so that $\Omega_{\opt}:=\Omega_\ast$ and $\mu_{\Omega_{\opt}}=\mu_{\partial\Omega_\ast}$ are as desired. See~\cite{MAGOULES-2020} for more details.
\end{proof}

\section{Relaxation and optimal shapes realizing the energy minimum}\label{S:uni}

Choosing a larger class of admissible shapes one can ensure the existence of an optimal shape which actually realizes the minimum  
of~\eqref{EqJ}. Following~\cite{HINZ-2020} we relax the restriction on the domains $\Omega$ to be Lipschitz domains and the restriction on the measures on the boundaries to be Hausdorff measures. 

Let $D\subset \mathbb{R}^N$ be a bounded Lipschitz domain and $\varepsilon>0$. By $\hat{\mathcal{O}}(D,\varepsilon)$ we denote the collection of all $(\varepsilon,\infty)$-domains $\Omega\subset D$. We require the measures on the boundary to satisfy the upper regularity condition \eqref{E:upperreg} and a second, somewhat complementary scaling condition. Given $s>0$ and a Borel measure  $\mu_F$ on $\mathbb{R}^N$ with $F=\supp \mu_F$ we say that $\mu_F$ is \emph{lower $s$-regular in the closed ball sense} if there is a constant $\overline{c}_s>0$ such that 
\begin{equation}\label{E:lowerreg}
\mu_F(\overline{B(x,r)})\geq \bar{c}_s r^s,\quad x\in F,\quad 0<r\leq 1.
\end{equation}

Now let $D_0$ be a non-empty Lipschitz domain and a proper subset of $D$ and suppose that $\Gamma_{\mathrm{Dir}}\subset \partial D\cap \partial D_0$ has positive measure $\mathcal{H}^{N-1}(\Gamma_{\mathrm{Dir}})>0$. Given $\varepsilon>0$, $N-1\leq s<N$, $0\leq d\leq s$, $\bar{c}_s>0$ and $c_d>0$ let 
\begin{multline}\label{E:U}
U_{\ad}(D,D_0,\varepsilon, c_v, s,d,\bar{c}_s,c_d,\Gamma_{\mathrm{Dir}}):=\{(\Omega,\mu_{\partial\Omega})\ |\ \Omega\in \hat{\mathcal{O}}(D,\varepsilon),\  D_0\subset \Omega,\\ \lambda^N(\Omega)=c_v,\
\partial\Omega=\Gamma_{\mathrm{Dir}}\cup \Gamma_{\mathrm{Neu}},\ \Gamma_{\mathrm{Dir}}= \partial \Omega\cap \partial D,\
\mu_{\partial\Omega}=\mathcal{H}^{N-1}|_{\Gamma_{\mathrm{Dir}}} + \mu_{\Gamma_{\mathrm{Neu}}}, \\ \text{$\mu_{\Gamma_{\mathrm{Neu}}}$ with $\supp \mu_{\Gamma_{\mathrm{Neu}}}= \Gamma_{\mathrm{Neu}}$ satisfies \eqref{E:upperreg} and \eqref{E:lowerreg}},\ \mu_{\partial\Omega}(\Gamma_{\mathrm{Dir}}\cap \Gamma_{\mathrm{Neu}})=0\}
\end{multline}

Using~\cite[Theorem 3]{HINZ-2020} it can be seen that $U_{\ad}(D,D_0,\varepsilon, c_v, s,d,\bar{c}_s,c_d,\Gamma_{\mathrm{Dir}})$ is compact with respect to the convergence in the Hausdorff sense, the sense of compacts, the sense of characteristic functions and the sense of weak convergence of the measures $\mu_{\partial\Omega}$ on the boundaries. By Lemma \ref{L:equivnorms} (ii) and Theorem \ref{ThWelPos} we have \eqref{EqEstimAPriori} with a uniform constant for all $(\Omega,\mu_{\partial\Omega})\in U_{\ad}(D,D_0,\varepsilon, c_v, s,d,\bar{c}_s,c_d,\Gamma_{\mathrm{Dir}})$. This, together with uniform estimates for the norms of extension operators, allows to obtain the following result by similar arguments as in the Lipschitz case.

\begin{theorem}\label{ThOptimalEpsDelta}
 Let $U_{\ad}(D,D_0,\varepsilon, c_v, s,d,\bar{c}_s,c_d,\Gamma_{\mathrm{Dir}})$ be as in \eqref{E:U}, $c_1\geq 0$ and $c_2\geq 0$. Let $\mathbf{f}\in L_2(D)^N$ and $\mathbf{g}\in  [\mathrm{Tr}(W^{1,2}(D))]^N$. 

Then there is $(\Omega_{\opt},\mu_{\partial\Omega_{\opt}})\in U_{\ad}(D,D_0,\varepsilon, c_v, s,d,\bar{c}_s,c_d,\Gamma_{\mathrm{Dir}})$ such that 
\begin{multline*}
J(\Omega_{\opt}, \mu_{\partial\Omega_{\opt}}, \u(\Omega_{\partial\Omega_{\opt}},\mu_{\partial\Omega_{\opt}}))\\
=\min_{(\Omega,\mu_{\partial\Omega})\in U_{\ad}(D,D_0,\varepsilon, c_v, s,d,\bar{c}_s,c_d,\Gamma_{\mathrm{Dir}})}J(\Omega, \mu_{\partial\Omega}, \u(\Omega,\mu_{\partial\Omega})).
\end{multline*}
\end{theorem}

Slight modifications of Theorems~\ref{ThPrincipale} and~\ref{ThOptimalEpsDelta} yield generalizations of Theorem \ref{Throof} on optimal roof shapes, such as the following. Let $N=3$ and let $D\subset \mathbb{R}^3$ be a bounded Lipschitz domain, $G$ a nonempty open proper subset of $D$ and $\Gamma_{\mathrm{Dir}}$ a compact subset of $\partial D$ with $\mathcal{H}^2(\Gamma_{\mathrm{Dir}})>0$ as in Section \ref{SecShapeOp}. Consider the class

\begin{multline}\label{E:Uroof}
	U_{\ad}(D,G,\varepsilon,c_v,s,d,\bar{c}_s,c_d,\Gamma_{\mathrm{Dir}})
	=\{(\Omega,\mu_{\partial\Omega})\ |\  \Omega\in \hat{\mathcal{O}}(D,\varepsilon),\ \lambda^3(\Omega)=c_\nu,\\ 
	\partial\Omega=\Gamma_{\mathrm{Dir}}\cup \Gamma_{\mathrm{Neu}},\ \Gamma_{\mathrm{Dir}}= \partial \Omega\cap \partial D, \ \mu_{\partial\Omega}=\mathcal{H}^{N-1}|_{\Gamma_{\mathrm{Dir}}} + \mu_{\Gamma_{\mathrm{Neu}}}, \\ \text{$\mu_{\Gamma_{\mathrm{Neu}}}$ with $\supp \mu_{\Gamma_{\mathrm{Neu}}}= \Gamma_{\mathrm{Neu}}$ satisfies \eqref{E:upperreg} and \eqref{E:lowerreg}},\ \mu_{\partial\Omega}(\Gamma_{\mathrm{Dir}}\cap \Gamma_{\mathrm{Neu}})=0,\\	
	\Gamma_{\mathrm{Neu}}=\Gamma_{\lo}\cup \Gamma_{\up},\ \Gamma_{\lo}\subset \overline{G}, \ \Gamma_{\up}=\Gamma_{\lo}+h_z\mathbf{e}_z\}.
\end{multline}

This class is compact with respect to the same four types of convergences as \eqref{E:U}. (We point out that 
\cite[Theorem 3]{HINZ-2020} can be used without problems, because the domains $\Omega$ 
are nonempty by construction.) The following result, which roughly speaking is a corollary of Theorem \ref{ThOptimalEpsDelta}, generalizes Theorem \ref{Throof} to the case of possibly nonzero Neumann data $\mathbf{g}$ on the moving parts. 

\begin{corollary}\label{C:roof}
Let $U_{\ad}(D,G,\varepsilon,c_v,s,d,\bar{c}_s,c_d,\Gamma_{\mathrm{Dir}})$ be as in \eqref{E:Uroof}, $\rho_0>0$, $c_1\geq 0$ and $c_2\geq 0$. For each $(\Omega,\mu_{\partial\Omega})\in U_{\ad}(D,G,\varepsilon,c_v,s,d,\bar{c}_s,c_d,\Gamma_{\mathrm{Dir}})$ let $\u(\Omega,\mu_{\partial\Omega})$ denote the unique weak solution to \eqref{Eq} with $\mathbf{g}\in [\mathrm{Tr}(W^{1,2}(D))]^3$ and $\mathbf{f}\equiv \mathbf{f}(\Omega)$ defined as in \eqref{Eqhz}. 

Then 
there is an optimal shape $(\Omega_{\opt},\mu_{\partial\Omega_{\opt}})\in U_{\ad}(D,G,\varepsilon,c_v,s,d,\bar{c}_s,c_d,\Gamma_{\mathrm{Dir}})$ realizing the minimum of~\eqref{EqJ} over $U_{\ad}(D,G,\varepsilon,c_v,s,d,\bar{c}_s,c_d,\Gamma_{\mathrm{Dir}})$.
\end{corollary}

\appendix
\section{Green's formulas}\label{AppDef}

We provide versions of Green's formulas in the context of Theorem \ref{ThTracecheap}. They are not used explicitely in our results, but they provide the correct justification of the variational formulation \eqref{Eq}. Consider the space
\[H(\mathrm{div},\Omega)=\{\u=(u_1,\ldots,u_N)^t:\Omega \to \R^N|\; u_1, \ldots u_N\in L^2(\Omega),\ \mathrm{div}\,\u \in L^2(\Omega)\}, \]
it is a Hilbert space corresponding to the norm
\[\|\u\|_{\mathrm{div},\Omega}=\left(\|\u\|^2_{L^2(\Omega)^N}+\|\div\, \u \|^2_{L^2(\Omega)} \right)^\frac{1}{2}.\]

\begin{proposition}\label{P:Greenformula}

		Let $\Omega\subset \mathbb{R}^N$ be a bounded $W^{1,2}$-extension domain. Suppose that $\mu_{\partial\Omega}$ is a Borel measure with $\supp \mu_{\partial\Omega}=\partial\Omega$ and such that \eqref{E:upperreg} holds with some $N-2<d\leq N$. 
	For all  $\mathbf{u}\in H(\mathrm{div},\Omega)$  we can define a bounded linear functional
	$\mathbf{u}\cdot \mathbf{n}\in (\operatorname{Tr}(W^{1,2}(\Omega)))'$ by the identity 
	\begin{equation}\label{EqGreenVect}
		\langle \mathbf{u}\cdot \mathbf{n}, \operatorname{Tr} w\rangle_{(\operatorname{Tr}(W^{1,2}(\Omega)))', \operatorname{Tr}(W^{1,2}(\Omega))}=\int_\Omega \mathbf{u}\cdot \nabla w\, \dx -\int_\Omega (\mathrm{div}\, \u) w\, \dx,
	\end{equation}
	$w\in W^{1,2}(\Omega)$.
	
For all $T\in H(\mathrm{div},\Omega)^N$ we can define a bounded linear functional $T\cdot \mathbf{n}\in [(\operatorname{Tr}(W^{1,2}(\Omega)))']^N$ by
\begin{equation}\label{EqGreen}
	\langle T\cdot \mathbf{n}, \operatorname{Tr}\theta\rangle_{[(\operatorname{Tr}(W^{1,2}(\Omega)))']^N, [\operatorname{Tr}(W^{1,2}(\Omega))]^N}=\int_\Omega T: \nabla \theta\, \dx -\int_\Omega (\mathrm{div}\, T)\cdot \theta\, \dx,
\end{equation}
$\theta\in W^{1,2}(\Omega)^N$. 
\end{proposition}

For the special case that $\Omega$ is a Lipschitz domain and $\mu_{\partial\Omega}$ equals $\mathcal{H}^{N-1}$ a proof of formula \eqref{EqGreenVect} can be found in \cite[Theorem 2.5 and formula (2.17)]{GIRAULT-1986}. The same proof, combined with Theorem \ref{ThTracecheap}, yields \eqref{EqGreenVect}. (Note also that a particular version of~\eqref{EqGreenVect} had been proved in~\cite{ARXIV-CREO-2018} for a specific domain $\Omega$.) Formula~\eqref{EqGreen} is just its matrix form and an immediate consequence of \eqref{EqGreenVect}. See~\cite[p. 69]{CIARLET-1988} for a version of this formula for smooth domains and tensors.

\section{A technical remark}\label{S:technical}

We comment briefly on how to see the claimed identity \eqref{E:itreallyvanishes} in the proof of Lemma \ref{L:equivnorms}, i.e., that $\widetilde{E_\Omega \u}=0$ $\mu_{\partial D}$-a.e. on $\Gamma_{\mathrm{Dir}}$. One can argue as follows: If $u_i$ are the components of $\u$ then the vector fields $\u^{(M)}$, $M\geq 1$, with components $u_i^{(M)}:=(-M)\vee (u_i\wedge M)$ satisfy $\lim_{M\to \infty}\big\|\u^{(M)}-\u\big\|_{W^{1,2}(\Omega)^{N}}=0$, what by \eqref{E:extensions} is inherited to their extensions. Therefore (and by Theorem \ref{ThTracecheap}) it suffices to show \eqref{E:itreallyvanishes} for $\u$ with bounded components. If $\chi$ is a smooth cut-off function with compact support inside an open ball $B\supset \overline{D}$ then we have $(\chi E_\Omega \u)^{\sim}=\widetilde{E_\Omega \u}$ on $\overline{D}$. Since  
$L^\infty$-functions with finite Dirichlet integral form an algebra and Poincar\'e's inequality holds for $B$, the product  $\chi E_\Omega \u$ is an element of $\mathring{W}^{1,2}(B)$, and we can use standard capacity arguments based on this space instead of $W^{1,2}(\mathbb{R}^N)$. In particular, adequate modifications of \cite[Theorems 5.2 and 5.4]{BIEGERT-2009} now allow to conclude
\eqref{E:itreallyvanishes} similarly as in \cite[Theorem 6.1]{BIEGERT-2009}.

\section{On the Mosco convergence of energy functionals}\label{SecMosco}

In this auxiliary section we consider energy functionals for Robin problems of type
\[\begin{cases}
	\hspace{30pt}-\mathrm{div} \:e(\mathbf{u})&=\mathbf{f} \quad \text{in $\Omega$},\\  
	e(\mathbf{u})\cdot n +\alpha \mathbf{1}_{\Gamma}\mathbf{u}&=0  \quad \text{on $\partial\Omega$},
	\end{cases}\]
where $\alpha\ge0$ and $\Gamma$ is a subset of $\partial\Omega$. We show that if a sequence of domains converges in a suitable sense then these energy functionals on the domains converge in the sense of Mosco \cite{MOSCO-1969}. 

Let $H$ be a Hilbert space. Recall from \cite[Definition 2.1.1]{MOSCO-1969} that a sequence of quadratic forms $a_n(\cdot,\cdot):H\times H\rightarrow (-\infty,+\infty]$ is said to \emph{$M$-converge} on $H$ to a quadratic form $a(\cdot,\cdot):H\times H\rightarrow (-\infty,+\infty]$ if
\begin{enumerate}
\item For every $u\in H$ there is a sequence $(u_n)_n\subset H$, convergent to $u$ and such that
$\varlimsup_{n\to\infty} a_n(u_n,u_n)\leq a(u,u)$.
\item For every sequence $(v_n)_n\subset H$ converging weakly to $u\in H$ we have
$\varliminf_{n\to\infty} a_n(v_n,v_n)\geq a(u,u)$.
\end{enumerate}

Let  $N\geq 2$ and let $\Omega\subset \mathbb{R}^N$ be a bounded domain. Suppose that $\mu_{\Gamma_{\mathrm{Dir}}}$ and $\mu_{\Gamma}$ are Borel measures with $\supp \mu_{\Gamma_{\mathrm{Dir}}}=\Gamma_{\mathrm{Dir}}$ and $\supp \mu_{\Gamma}=\Gamma$ such that $\partial\Omega=\Gamma_{\mathrm{Dir}}\cup \Gamma$ and $\Gamma_{\mathrm{Dir}}\cap \Gamma$ is a zero set for both $\mu_{\Gamma_{\mathrm{Dir}}}$ and $\mu_{\Gamma}$. For each $n\geq 1$ let $\Omega_n\subset \mathbb{R}^N$ a domain and let $\mu_{\Gamma_n}$ be a Borel measure with $\supp \mu_{\Gamma_n}=\Gamma_{n}$ such that $\partial\Omega_n=\Gamma_{\mathrm{Dir}}\cup \Gamma_{n}$ and $\Gamma_{\mathrm{Dir}}\cap \Gamma_{n}$ is a zero set for both $\mu_{\Gamma_{\mathrm{Dir}}}$ and $\mu_{\Gamma_n}$.
If all domains are $W^{1,2}$-extension domains contained in a bounded Lipschitz domain $D$ and all $\mu_{\Gamma_{\mathrm{Dir}}}$, $\mu_{\Gamma}$ and $\mu_{\Gamma_n}$ satisfy \eqref{E:upperreg} with $N-2<d\leq N$ 
then for any non-negative real numbers $\alpha$ and $\alpha_n$ we can define quadratic forms on $L^2(D)^N$ by

\begin{align}
	a_n(\u,&\u)=\nonumber\\
	&\left\{\begin{array}{l}
	             	\int\limits_{\Omega_n} e(\u):e(\u)\dx+\int\limits_{\Omega_n}|\u|^2\dx +\alpha_n \int\limits_{\Gamma_{n}}|\mathrm{Tr} \u|^2 d \mu_{\Gamma_n}, \; \u\in V(\Omega_n,\Gamma_{\mathrm{Dir}})^N,\\ 
	             	+\infty,\;\; \hbox{otherwise}
	             \end{array}
	             \right. \notag\\
	             \text{and} \notag\\
    a(\u,&\u)=\nonumber\\
	&\left\{\begin{array}{l}
	             	\int\limits_{\Omega} e(\u):e(\u)\dx+\int\limits_{\Omega}|\u|^2\dx +\alpha \int\limits_{\Gamma}|\mathrm{Tr} \u|^2 d \mu_{\Gamma}, \; \u\in V(\Omega,\Gamma_{\mathrm{Dir}})^N,\\ 
	             	+\infty,\;\; \hbox{otherwise.}
	             \end{array}
	             \right. \notag
\end{align}

The following convergence result can be obtained by the same arguments as ~\cite[Theorem~8]{HINZ-2020}.

\begin{theorem}\label{ThMoscoR}

Let $\varepsilon>0$, let $(\Omega_n)_n$ be a sequence of $(\varepsilon,\infty)$-domains $\Omega_n\subset D$. Let $\Gamma_{\mathrm{Dir}}$, $\Gamma_n$ and $\mu_{\Gamma_{\mathrm{Dir}}}$, $\mu_{\Gamma_n}$ be as above and assume that all $\mu_{\Gamma_n}$ satisfy \eqref{E:upperreg} with the same exponent $N-1<d\leq N$ and the same constant $c_d$. 

If $\lim_n \Omega_n=\Omega$ in the Hausdorff sense and in the sense of characteristic functions and $\lim_n \alpha_n\mu_n=\alpha\mu$ in the sense of weak convergence, then we have 
\[\lim_n a_n=a\]
in the sense of $M$-convergence on $L^2(D)$.
	\end{theorem}
	
\begin{remark}
For simplicity, and because it is sufficient for many practical purposes, $\Gamma_{\mathrm{Dir}}$ is assumed to be fixed in the above setup. It would actually be sufficient to keep the set $\Gamma_{\mathrm{Dir}}\cap \Gamma_n$ fixed, see ~\cite[Theorem~10]{DEKKERS-2020}.
\end{remark}

%

\def\refname{References}
\bibliographystyle{siam}
\bibliography{biblio.bib}
\end{document}